\documentclass[amsthm]{autart}    % Enable this line and disable the 
\usepackage{color}
\usepackage{subcaption}%this is to put multiply plots in one figure
\usepackage{graphicx}%for multiple plots
\usepackage{amsmath,amsfonts}
\usepackage{bm}%bolds greek symbols
\usepackage{soul}%to strikethrough
\usepackage{sidecap}%for side captions in figures
\usepackage{ifthen}
\usepackage{comment}
\begin{document}
%%-----------------------------
%%      the top matter
%%-----------------------------
\begin{frontmatter}\vspace*{0in}
\title{Asymptotic Stability of  the Landau--Lifshitz Equation}%\thanks{...}\thanks{...}% At most 5 thanks
\author{Amenda Chow}\address{Department of Mathematics and Statistics, York University, Canada\\ amchow@yorku.ca}
 
          
\begin{abstract}
The Landau--Lifshitz equation describes the behaviour of magnetic domains in ferromagnetic structures. Recently such structures have been found to be favourable for storing digital data. Stability of magnetic domains is important for this. Consequently, asymptotic stability of the equilibrium points in the Landau--Lifshitz equation are established.  A suitable Lyapunov function is presented. 
\end{abstract}

\end{frontmatter}

\section{Introduction}
The Landau--Lifshitz equation is a coupled set of nonlinear partial differential equations. One of its first appearances is in a 1935 paper \cite{Landau1935}, in which this equation describes the behaviour of magnetic domains within a ferromagnetic structure. In recent applications, structures such as ferromagnetic nanowires have appeared in electronic devices used for storing digital information \cite{Parkin2008}.  In particular, data is encoded as a specific pattern of magnetic domains within a ferromagnetic nanowire. Consequently the Landau--Lifshitz equation continues to be widely explored, and its stability is of growing interest \cite{CarbouGilles2010,CarbouLabbe2006,Carbou2011,Chow2016,Gou2011,GuoBook2008,Guo2004,Jizzini2011,Labbe2012,Mayergoyz2010,Zhai1998}. Stability of equilibrium points is also related to hysteresis \cite{Chow2014_ACC,Morris2011}, and investigating stability lends insights into the hystereric behaviour that appears in the Landau--Lifshitz equation \cite{CarbouEfendiev2009,Chow2014_ACC,Wiele2007,Visintin1997}. 

Stability results are often based on linearization \cite{CarbouLabbe2006,Gou2011,Labbe2012}; however, in the proceeding discussion, asymptotic stability of the Landau-Lifshitz equation is established using Lyapunov theory. This is preferred because linearization leads only to an approximation.

The difficulty with Lyapunov theory is often in the construction of an appropriate Lyapunov function.  Working in infinite-dimensions makes this more difficult; however, Lyapunov functions have been found for the Landau--Lifshitz equation \cite{Chow2016,GuoBook2008}.  In both these works, a Lyapunov function establishes that the equilibrium points of the Landau-Lifshitz equation are stable. The work in \cite{Chow2016} is extended here and asymptotic stability is shown. In particular, a nonlinear control is shown to steer the system to an asymptotically stable equilibrium point.  Control of the Landau--Lifshitz equation is crucial as this means the behaviour of the domain walls, which contains the encoded data, can be fully determined \cite{Parkin2008}. 

The control objective is to steer the system dynamics to any arbitrary equilibrium point, which requires asymptotic stability. This is presented in Theorem~\ref{thmasymstab}, which is the main result and can be found in Section~\ref{secsymstab}.  A summary and future avenues are in the last section. To begin, a brief mathematical review of the Landau--Lifshitz equation is discussed next.

%%%%%%%%%%%%%%%
\section{Landau-Lifshitz Equation}
%%%%%%%%%%%%%%%
Consider a one dimensional ferromagnetic nanowire of length $L>0$. Let   $\mathbf m(x,t)=(m_1(x,t),m_2(x,t),m_3(x,t))$ be the magnetization of the ferromagnetic nanowire for some position $x \in [0,L]$  and time $t \geq 0 $ with initial magnetization $\mathbf m(x,0)=\mathbf m_0(x)$.  These dynamics are determined by  
\begin{align}
&\frac{\partial \mathbf m}{\partial t}  = \mathbf m \times \left( \mathbf m_{xx}+\mathbf u\right)-\nu\mathbf m\times\left(\mathbf m\times \left(\mathbf m_{xx}+\mathbf u\right)\right) \label{eqcontrolledLLphysical}\\
& \mathbf m_x(0,t)=  \mathbf m_x(L,t)=\mathbf 0, \label{eqboundarycondition}\\
&  || \mathbf m(x,t)||_{2} =1.\label{eqconstraint}
\end{align}
where $\times$ denotes the cross product and $||\cdot||_{2}$ is the Euclidean norm. Equation~(\ref{eqcontrolledLLphysical}) is the one--dimensional controlled Landau--Lifshitz equation \cite{Bertotti2009,Chow2014_ACC,Chow2016,Gilbert2004,GuoBook2008}.  It satisfies the constraint in (\ref{eqconstraint}), which means the magnitude of the magnetization is uniform at every point of the ferromagnet. The exchange energy is $\mathbf m_{xx}$. Mathematically, $\mathbf m_{xx}$ denotes magnetization differentiated with respect to $x$ twice. The parameter $\nu \geq  0$ is the damping parameter, which depends on the type of ferromagnet. The applied magnetic field, denoted $\mathbf u(t)$, acts as the control, and hence, when $\mathbf u(t)=\mathbf 0$, equation~(\ref{eqcontrolledLLphysical})  can be thought of as the uncontrolled Landau--Lifshitz equation. Neumann boundary conditions  (\ref{eqboundarycondition}) are used here. 

Existence and uniqueness results  can be found in \cite{Alouges1992,Carbou2001,Chow2013_thesis,Chow2016} and references therein.  Solutions to (\ref{eqcontrolledLLphysical}) are defined on $\mathcal L_2^3 = \mathcal L_2 ([0,L]; \mathbb R^3)$ with the usual inner--product and norm with domain
\begin{align*}
%\label{setDforfullLL}
D=\{ & \mathbf m\in \mathcal L_2^3 :  \mathbf m_x \in \mathcal L_2^3, \, \\
& \mathbf m_{xx} \in \mathcal L_2^3,  \mathbf m_x(0)=\mathbf m_x(L) = \mathbf 0    \}.  
\end{align*}
The notation $\|\cdot\|_{\mathcal L_2^3}$ is used for the norm.

%%%%%%%%%%%%%%%
%ASYMPTOTIC STABILITY
%%%%%%%%%%%%%%
\section{Asymptotic Stability}\label{secsymstab}
For $\mathbf u(t)=\mathbf 0$, the set of equilibrium points is
\begin{align}\label{equilibriumset}
E=&\{\mathbf a=(a_1,a_2,a_3) : a_1,a_2,a_3\nonumber\\
& \mbox{ constants and }||\mathbf a||_2=1 \},
\end{align} 
which satisfies the boundary conditions in (\ref{eqboundarycondition}) \cite[Theorem~6.1.1]{GuoBook2008}.  It is clear $E$ contains an infinite number of equilibria. A particular $\mathbf a \in E$ is stable but not asymptotically stable \cite[Proposition~6.2.1]{GuoBook2008}; however, the set $E$ is asymptotically stable in the $\mathcal L_2^3$--norm \cite{Chow2013_thesis,Chow2016}. 

Let $\mathbf r$ be an arbitrary equilibrium point of $E$ with $r_1\neq 0$ to ensure $||\mathbf r||_2=1$, that is,  (\ref{eqconstraint}), is  satisfied. Define the control in (\ref{eqcontrolledLLphysical}) to be  
% (See Figure~\ref{figblockdiagram}.) 
\begin{equation}\label{equ}
\mathbf u=k(\mathbf r -\mathbf m )
\end{equation}
where  $k$ is a scalar constant. This is the same control used in \cite{Chow2016}.  
For this control, $\mathbf r$ is an equilibrium point of the controlled Landau-Lifshitz equation (\ref{eqcontrolledLLphysical}). It is shown in Theorem~\ref{thmasymstab} that $\mathbf r$ is locally asymptotically stable, which is the main result.   Simulations demonstrating asymptotic stability of the Landau-Lifshitz equation given the control in (\ref{equ}) are shown in \cite{Chow2016}.

The following lemmas are needed in Theorem~\ref{thmasymstab}. Lemmas~\ref{thmderivativemcrossmprime} and~\ref{lemmazerointegral} appear in \cite{Chow2016}, which can be obtained from the product rule and integration by parts, respectively. 
\begin{lem}\label{thmderivativemcrossmprime}
For $\mathbf m\in \mathcal L_2^3$, the derivative of $\mathbf g=\mathbf m \times \mathbf m_x$ is $\mathbf g_x=\mathbf m \times \mathbf m_{xx}$.
\end{lem}

\begin{lem}\label{lemmazerointegral}
For $\mathbf m \in \mathcal L_2^3$ satisfying (\ref{eqboundarycondition}), 
\[
\int_0^L (\mathbf m-\mathbf r)^{\mathrm T}(\mathbf m \times \mathbf m_{xx})dx=0.
\]
\end{lem}

\begin{lem}\label{lemupperbound1}
If $\mathbf r \in E$ and $\mathbf m$ satisfies (\ref{eqconstraint}), then $||\mathbf m \times \mathbf r||_2\leq1 $.
\end{lem}

\begin{proof}
Recall $||\mathbf m \times \mathbf r||_2 = ||\mathbf m||_2||\mathbf r ||_2\sin(\theta)$ where $\theta$ is the angle between $\mathbf m$. Since $\mathbf r$ and $||\mathbf m||_2=||\mathbf r||_2=1$, then
$||\mathbf m \times \mathbf r||_2 = ||\mathbf m||_2||\mathbf r ||_2\sin(\theta) \leq \sin(\theta)\leq 1$.
\end{proof}

\begin{thm}\label{thmasymstab}
For any $\mathbf r \in E$, there exists a range of positive values of $k$ such that $\mathbf r$ is a locally asymptotically stable equilibrium point of (\ref{eqcontrolledLLphysical})  in the $\mathcal L_2^3$--norm. 
\end{thm}

\begin{proof}
Let $B(\mathbf r,p)=\{\mathbf m \in \mathcal L_2^3: ||\mathbf m -\mathbf r||_{ \mathcal L_2^3}<p \} \subset D$ for some constant $0<p<2$.  Note that since $p<2$, then  $-\mathbf r \notin B(\mathbf r,p)$. For any $\mathbf m \in B(\mathbf r,p)$, the Lyapunov function is 
\[
V(\mathbf m)=\frac{f(k)}{2}\left| \left| \mathbf m-\mathbf r\right|\right|_{\mathcal L_2^3}^2+\frac{1}{2}\left| \left|  \mathbf m_x\right|\right|_{\mathcal L_2^3}^2
\]
where $f(k)>0$ is a scalar function of $k$ and $|f(k)+k|\leq1$ for all $k>0$. Such functions exist. For example, $f(k)=k$ for $k\in(0,1/2]$.  

Taking the derivative of $V$,
\begin{align}
\frac{dV}{dt}&=\int_0^Lf(k)(\mathbf m -\mathbf r)^{\mathrm T}\dot{{\mathbf m}} dx+\int_0^L\mathbf m_x^{\mathrm T} \dot{{\mathbf m}}_xdx \nonumber\\
&=\int_0^Lf(k)(\mathbf m -\mathbf r)^{\mathrm T}\dot{\mathbf m} dx-\int_0^L\mathbf m_{xx}^{\mathrm T} \dot{\mathbf m}dx\nonumber\\
&=\int_0^L\left(f(k)(\mathbf m -\mathbf r)^{\mathrm T}\dot{\mathbf m} -\mathbf m_{xx}^{\mathrm T} \dot{\mathbf m}\right)dx\label{eqLyapunovFunction}
\end{align}
where the dot notation means differentiation with respect to $t$. 

Letting $\mathbf h = \mathbf m - \mathbf r$, the integrand in (\ref{eqLyapunovFunction}) becomes
\begin{equation}\label{eqintegrand}
f(k)\mathbf h^{\mathrm T}\dot{\mathbf m} - \mathbf m_{xx}^{\mathrm T} \dot{\mathbf m}
\end{equation}
and equation~(\ref{eqcontrolledLLphysical}) becomes
\[
\dot{ \mathbf m}  = \mathbf m \times \left( \mathbf m_{xx}-k\mathbf h\right)-\nu\mathbf m\times\left(\mathbf m\times \left(\mathbf m_{xx}-k\mathbf h\right)\right).
\]
It follows that 
\begin{align}
\mathbf h^{\mathrm T}\dot{\mathbf m}  = &\mathbf h^{\mathrm T}  \left[\mathbf m \times \left( \mathbf m_{xx}-k\mathbf h\right)-\nu\mathbf m\times\left(\mathbf m\times \left(\mathbf m_{xx}-k\mathbf h\right)\right)\right] \nonumber\\
=& \mathbf h^{\mathrm T}  \left(\mathbf m \times \mathbf m_{xx}\right) -k \mathbf h^{\mathrm T}\left(\mathbf m \times \mathbf h\right) \nonumber\\
&-\nu \mathbf h^{\mathrm T}\left[\mathbf m\times\left(\mathbf m\times \mathbf m_{xx}\right)\right] +\nu k  \mathbf h^{\mathrm T}\left[\mathbf m\times\left(\mathbf m\times \mathbf h\right)\right] \nonumber \\
=& \mathbf h^{\mathrm T}  \left(\mathbf m \times \mathbf m_{xx}\right) -\nu \left(\mathbf m\times \mathbf m_{xx}\right)^{\mathrm T}\left(\mathbf h\times \mathbf m\right)\nonumber\\
&+\nu k  \left(\mathbf m\times \mathbf h\right)^{\mathrm T}\left(\mathbf h\times\mathbf m\right)  \nonumber\\
=& \mathbf h^{\mathrm T}  \left(\mathbf m \times \mathbf m_{xx}\right) -\nu \left(\mathbf m\times \mathbf m_{xx}\right)^{\mathrm T}\left(\mathbf h\times \mathbf m\right) \nonumber\\
&-\nu k  ||\mathbf m\times \mathbf h||_2^2 \label{eqintegrandfirstterm}
\end{align}
 and 
\begin{align}
 \mathbf m_{xx}^{\mathrm T} \dot{\mathbf m} =&   \mathbf m_{xx}^{\mathrm T} \left[  \mathbf m \times \left( \mathbf m_{xx}-k\mathbf h\right)-\nu\mathbf m\times\left(\mathbf m\times \left(\mathbf m_{xx}-k\mathbf h\right)\right) \right] \nonumber\\
 =&  \mathbf m_{xx}^{\mathrm T} \left(  \mathbf m \times \mathbf m_{xx}\right) -k\mathbf m_{xx}^{\mathrm T} \left(  \mathbf m \times \mathbf h\right)\nonumber\\
 &-\nu\mathbf m_{xx}^{\mathrm T} \left[ \mathbf m\times\left(\mathbf m\times  \mathbf m_{xx} \right)\right]+\nu k\mathbf m_{xx}^{\mathrm T} \left[ \mathbf m\times\left(\mathbf m\times  \mathbf h\right)\right] \nonumber\\
 =&-k\mathbf m_{xx}^{\mathrm T} \left(  \mathbf m \times \mathbf h\right)-\nu\mathbf m_{xx}^{\mathrm T} \left[ \mathbf m\times\left(\mathbf m\times  \mathbf m_{xx} \right)\right] \nonumber\\
 &+\nu k\mathbf m_{xx}^{\mathrm T} \left[ \mathbf m\times\left(\mathbf m\times  \mathbf h\right)\right] \nonumber\\
 =&-k\mathbf m_{xx}^{\mathrm T} \left(\mathbf m \times \mathbf h\right)-\nu\left(\mathbf m\times  \mathbf m_{xx} \right)^{\mathrm T} \left(  \mathbf m_{xx}\times \mathbf m\right)\nonumber\\
 & +\nu k\left(\mathbf m\times  \mathbf h\right)^{\mathrm T} \left(\mathbf m_{xx} \times\mathbf m\right) \nonumber\\
  =&-k\mathbf m_{xx}^{\mathrm T} \left( \mathbf m \times \mathbf h\right)+\nu||\mathbf m\times  \mathbf m_{xx} ||_2^2 \nonumber\\
  &+\nu k\left(\mathbf m\times  \mathbf h\right)^{\mathrm T} \left(\mathbf m_{xx} \times\mathbf m\right). \label{eqintegrandsecondterm}
\end{align}

Substituting (\ref{eqintegrandfirstterm}) and (\ref{eqintegrandsecondterm}) into equation~(\ref{eqintegrand}) leads to
\begin{align*}
&f(k)\mathbf h^{\mathrm T}\dot{\mathbf m} - \mathbf m_{xx}^{\mathrm T} \dot{\mathbf m}\nonumber\\
=& f(k)\mathbf h^{\mathrm T}  \left(\mathbf m \times \mathbf m_{xx}\right) -\nu f(k) \left(\mathbf m\times \mathbf m_{xx}\right)^{\mathrm T}\left(\mathbf h\times \mathbf m\right) \nonumber\\
& -\nu kf(k)  ||\mathbf m\times \mathbf h||_2^2 +k\mathbf m_{xx}^{\mathrm T} \left(  \mathbf m \times \mathbf h\right)-\nu||\mathbf m\times  \mathbf m_{xx} ||_2^2\\
&-\nu k\left(\mathbf m\times  \mathbf h\right)^{\mathrm T} \left(\mathbf m_{xx} \times\mathbf m\right)\\
=& (f(k)-k)\mathbf h^{\mathrm T}  \left(\mathbf m \times \mathbf m_{xx}\right) -\nu f(k) \left(\mathbf m\times \mathbf m_{xx}\right)^{\mathrm T}\left(\mathbf h\times \mathbf m\right) \\
& -\nu kf(k)  ||\mathbf m\times \mathbf h||_2^2 -\nu||\mathbf m\times  \mathbf m_{xx} ||_2^2\\
&-\nu k\left(\mathbf m\times  \mathbf h\right)^{\mathrm T} \left(\mathbf m_{xx} \times\mathbf m\right)\\
=& (f(k)-k)\mathbf h^{\mathrm T}  \left(\mathbf m \times \mathbf m_{xx}\right) \\
&-\nu (f(k)+k) \left(\mathbf m\times \mathbf m_{xx}\right)^{\mathrm T}\left(\mathbf h\times \mathbf m\right)\\
&-\nu kf(k)  ||\mathbf m\times \mathbf h||_2^2 -\nu||\mathbf m\times  \mathbf m_{xx} ||_2^2
\end{align*}
Substituting this expression into equation~(\ref{eqLyapunovFunction}) leads to
\begin{align*}
\frac{dV}{dt}  =&(f(k)-k) \int_0^L\mathbf h^{\mathrm T}  \left(\mathbf m \times \mathbf m_{xx}\right)dx\\
&-\nu (f(k)+k)  \int_0^L \left(\mathbf m\times \mathbf m_{xx}\right)^{\mathrm T}\left(\mathbf m\times \mathbf h\right) dx\\
& - \nu kf(k)\int_0^L  ||\mathbf m\times \mathbf h||_2^2 dx-\nu \int_0^L ||\mathbf m\times  \mathbf m_{xx} ||_2^2 dx.
\end{align*}
From Lemma~\ref{lemmazerointegral}, the first integral equals zero since $\mathbf h = \mathbf m -\mathbf r$. It follows that 
\begin{align*}
\frac{dV}{dt}  =&-\nu (f(k)+k)  \int_0^L \left(\mathbf m\times \mathbf m_{xx}\right)^{\mathrm T}\left(\mathbf m\times \mathbf h\right) dx\\
& - \nu kf(k)\int_0^L  ||\mathbf m\times \mathbf h||_2^2 dx-\nu \int_0^L ||\mathbf m\times  \mathbf m_{xx} ||_2^2 dx\\
=& -\nu (f(k)+k) \int_0^L \left(\mathbf m\times \mathbf m_{xx}\right)^{\mathrm T}\left(\mathbf m\times \mathbf h\right) dx \\
&- \nu kf(k) ||\mathbf m\times \mathbf h||_{\mathcal L_2^3}^2-\nu  ||\mathbf m\times  \mathbf m_{xx} ||_{\mathcal L_2^3}^2.
\end{align*}
Applying the Cauchy-Schwarz Inequality to the integrand leads to
\begin{align*}
\frac{dV}{dt} \leq &\nu |f(k)+k| \int_0^L ||\mathbf m\times \mathbf m_{xx}||_2||\mathbf m\times \mathbf h||_2 dx \\
&- \nu kf(k) ||\mathbf m\times \mathbf h||_{\mathcal L_2^3}^2-\nu  ||\mathbf m\times  \mathbf m_{xx} ||_{\mathcal L_2^3}^2.
\end{align*}
Since $||\mathbf m\times \mathbf h||_2\leq 1$ from  Lemma~\ref{lemupperbound1}, then
\begin{align*}
\frac{dV}{dt} \leq &\nu |f(k)+k| \int_0^L ||\mathbf m\times \mathbf m_{xx}||_2dx - \nu kf(k) ||\mathbf m\times \mathbf h||_{\mathcal L_2^3}^2\\
&-\nu  ||\mathbf m\times  \mathbf m_{xx} ||_{\mathcal L_2^3}^2\\
=&\nu |f(k)+k| \,  ||\mathbf m\times \mathbf m_{xx}||_{\mathcal L_2^3}^2 - \nu kf(k) ||\mathbf m\times \mathbf h||_{\mathcal L_2^3}^2\\
&-\nu  ||\mathbf m\times  \mathbf m_{xx} ||_{\mathcal L_2^3}^2\\
=&\nu \left(|f(k)+k|-1\right) \,  ||\mathbf m\times \mathbf m_{xx}||_{\mathcal L_2^3}^2 - \nu kf(k) ||\mathbf m\times \mathbf h||_{\mathcal L_2^3}^2.
\end{align*}

Since $|f(k)+k|\leq1$, then
\begin{align}\label{eqdervativeofLyapunovfinal}
\frac{dV}{dt} &\leq - \nu kf(k) ||\mathbf m\times \mathbf h||_{\mathcal L_2^3}^2.
\end{align}
Since $k>0$ and $f(k)>0$, then $\frac{dV}{dt} \leq 0$ with $\frac{dV}{dt} =0$ if and only if $\mathbf m\times \mathbf h=\mathbf0$.  For example, suppose $f(k)=k$ on $k\in(0,1/2]$, which satisfies  $|f(k)+k|\leq1$. Equation~(\ref{eqdervativeofLyapunovfinal}) becomes
 \[
 \frac{dV}{dt} \leq - \nu k^2 ||\mathbf m\times \mathbf h||_{\mathcal L_2^3}^2
 \]
 which is clearly less than or equal to zero, and the value of $k$ can be any number in the interval $(0,1/2]$. For instance, picking $k=1/4$, the Lyapunov function is
 \[
 V(\mathbf m)=\frac{1}{8}\left| \left| \mathbf m-\mathbf r\right|\right|_{\mathcal L_2^3}^2+\frac{1}{2}\left| \left|  \mathbf m_x\right|\right|_{\mathcal L_2^3}^2.
 \]

Revisiting equation (\ref{eqdervativeofLyapunovfinal}), if $\mathbf m =\mathbf r$, then $\mathbf h=\mathbf 0$ and hence $dV/dt=0$. On the other hand, $dV/dt=0$ implies $\mathbf m\times  \mathbf h=\mathbf 0,$ and hence $\mathbf r\times  \mathbf m=\mathbf 0.$ This is a system of algebraic equations,
\begin{align*}
r_2m_3-r_3m_2&=0\\
r_3m_1-r_1m_3&=0\\
r_1m_2-r_2m_1&=0.
\end{align*}
The solution is $m_2=\displaystyle\frac{r_2}{r_1}m_1$ and $m_3=\displaystyle\frac{r_3}{r_1}m_1$ for any $m_1$ and $r_1\neq0$.  Given  $||\mathbf m||_2=1$ and $||\mathbf r||_2=1$, this leads to  $m_1^2=r_1^2$ and hence $m_1=\pm r_1$, which leads to $m_2=\pm r_2$ and $m_3=\pm r_3$; that is, $\mathbf m = \pm \mathbf r$.  For $V(\mathbf m)$ on $ B(\mathbf r,p)$, this implies $\mathbf m=\mathbf r$ if $dV/dt=0$. Local  asymptotic stability  follows from Lyapunov's Theorem \cite[Theorem~6.2.1]{Michel1995}.
\end{proof}

%%%%%%%%%%%%%%%
\section{Discussions}
%%%%%%%%%%%%%%%
Asymptotic stability of an arbitrary equilibrium point of the Landau--Lifshitz equation with Neumann boundary conditions is shown in Theorem~\ref{thmasymstab}. This is established using Lyapunov theory. The result in Theorem~\ref{thmasymstab} is an extension of the work presented in \cite{Chow2016}.  

The control given in (\ref{equ}) can be used to control  the hysteresis that often appears in magnetization dynamics including those described by the Landau--Lifshitz equation \cite{Chow2013_thesis,Chow2014_ACC}. Figure~\ref{fighysteresisLLphysicalcontrol} depicts the input-output map of the Landau--Lifshitz equation in (\ref{eqcontrolledLLphysical}). The output is the magnetization, $\mathbf m(x,t)=(m_1(x,t),m_2(x,t),m_2(x,t))$, and the input is a periodic function denoted $\mathbf{\hat{u}}(t)$ where $\omega$ is the frequency of this input. For each $m_i$ with $i=1,2,3$, the input is the periodic function $0.01\cos(\omega t)$. For this periodic input, equation (\ref{eqcontrolledLLphysical}) becomes
\[
\frac{\partial \mathbf m}{\partial t}  = \mathbf m \times \left( \mathbf m_{xx}+\mathbf u\right)-\nu\mathbf m\times\left(\mathbf m\times \left(\mathbf m_{xx}+\mathbf u\right)\right) +\mathbf{\hat{u}}(t).
\]
 As the frequency of the (periodic) input approaches zero, loops appear in the input--output map for $m_1(x,t)$, which indicates the presence of hysteresis \cite{Bernstein2005}.

\begin{figure}
\begin{tabular}{ccc}
\hspace*{-0.75cm} \includegraphics[height=3.6cm]{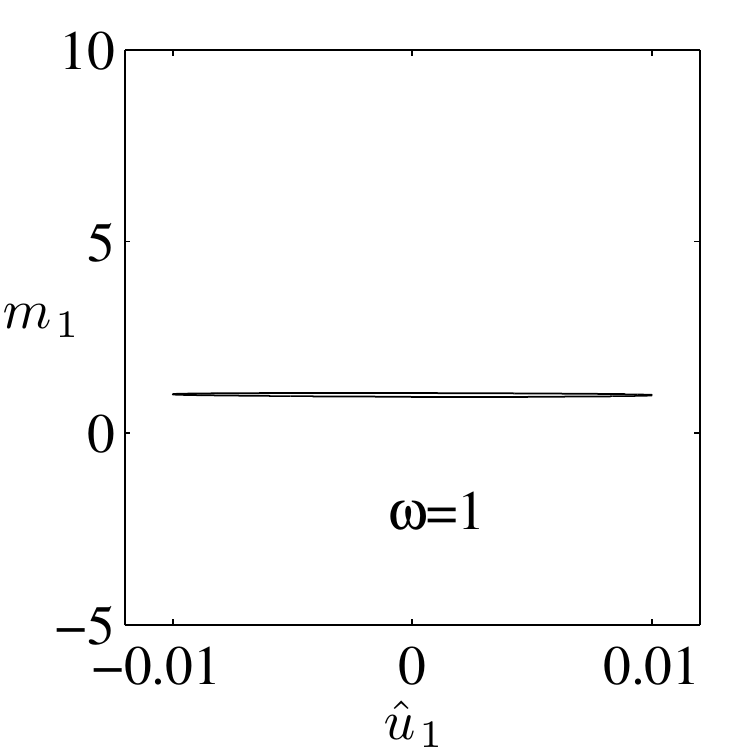}\hspace*{-0.6cm} &
  \includegraphics[height=3.6cm]{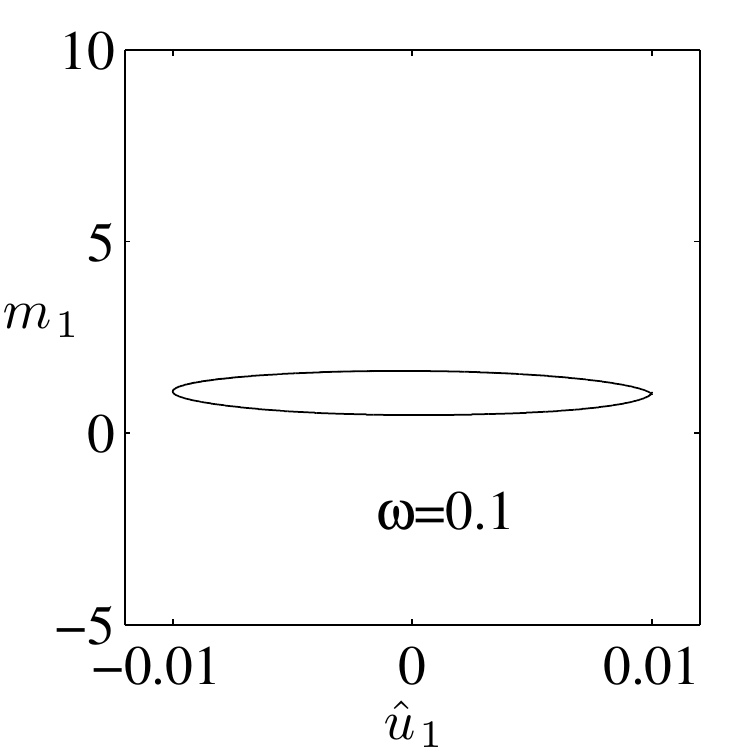}\hspace*{-0.6cm} &
   \includegraphics[height=3.6cm]{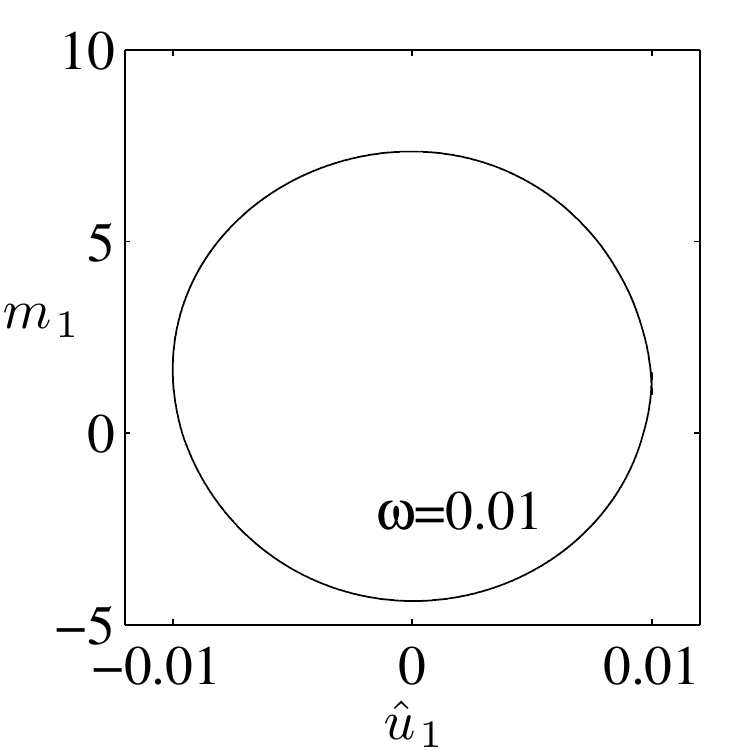}\vspace*{-0.2cm}\\
   \small{(a)}&\small{(b)}&\small{(c)}\\
 \hspace*{-0.75cm}  \includegraphics[height=3.6cm]{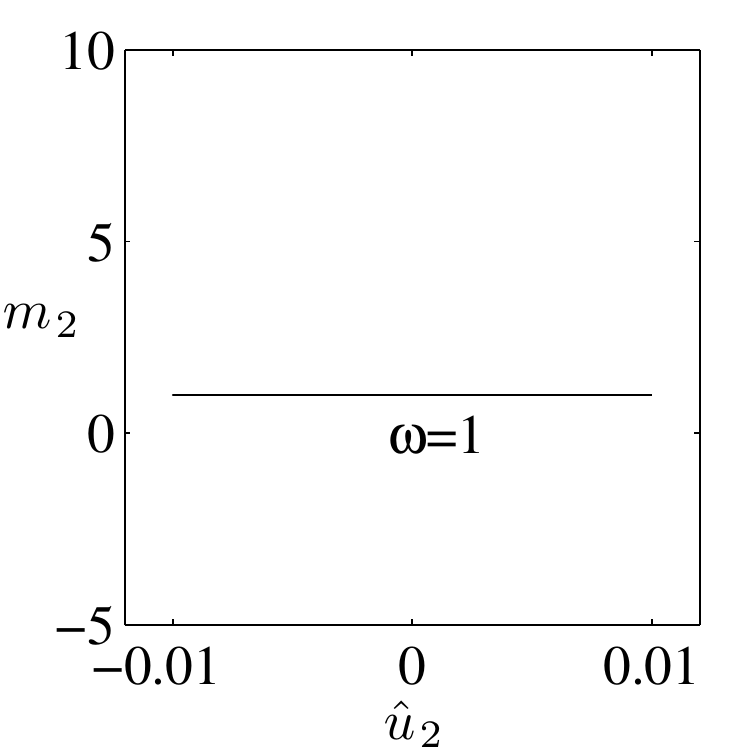}\hspace*{-0.6cm} &
  \includegraphics[height=3.6cm]{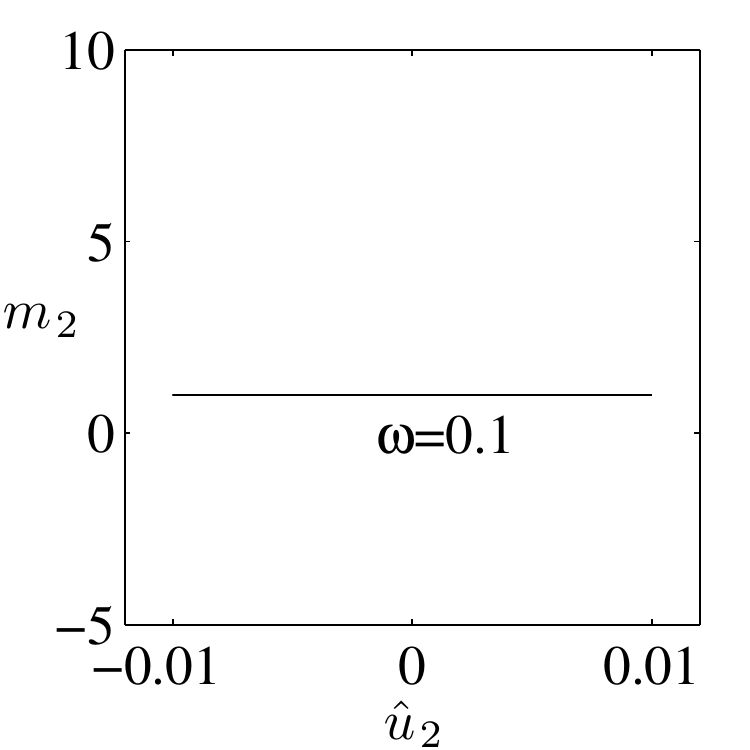}\hspace*{-0.6cm} &
   \includegraphics[height=3.6cm]{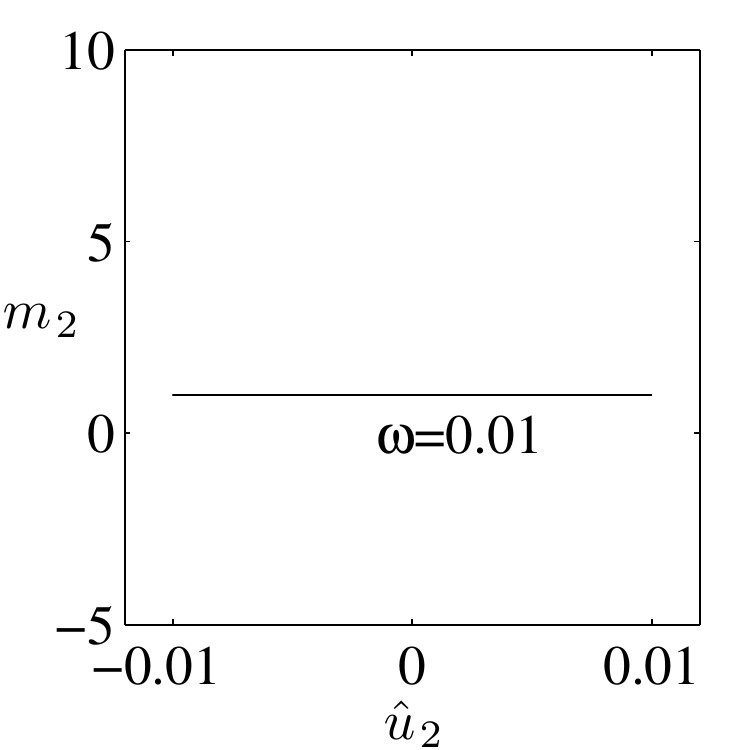}\vspace*{-0.2cm}\\
   \small{(d)}&\small{(e)}&\small{(f)}\\
    \hspace*{-0.75cm}  \includegraphics[height=3.6cm]{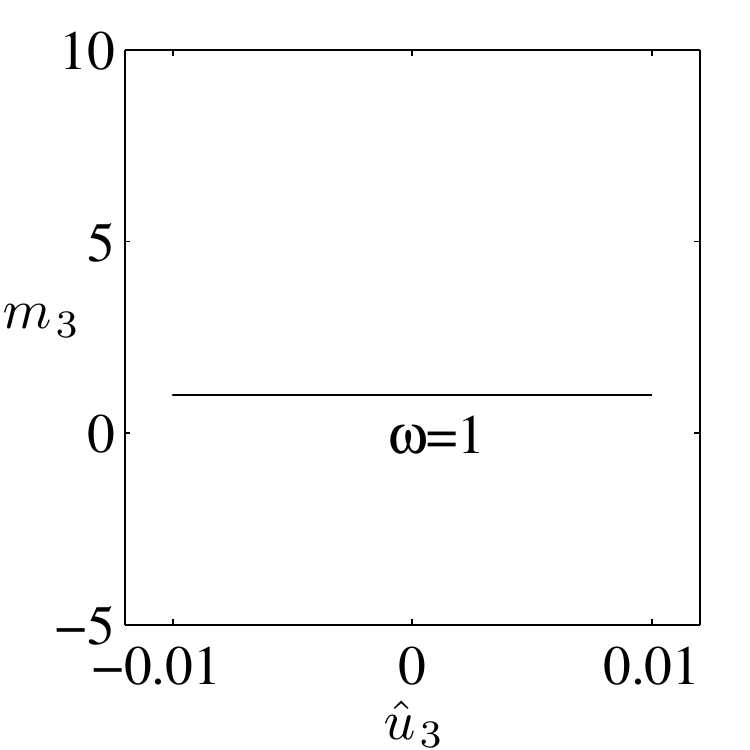}\hspace*{-0.6cm} &
  \includegraphics[height=3.6cm]{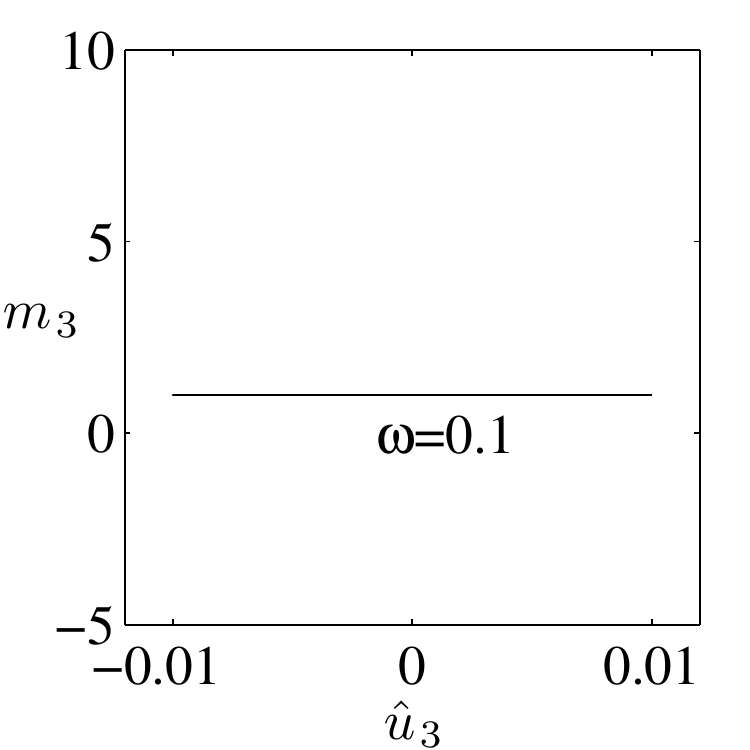}\hspace*{-0.6cm} &
   \includegraphics[height=3.6cm]{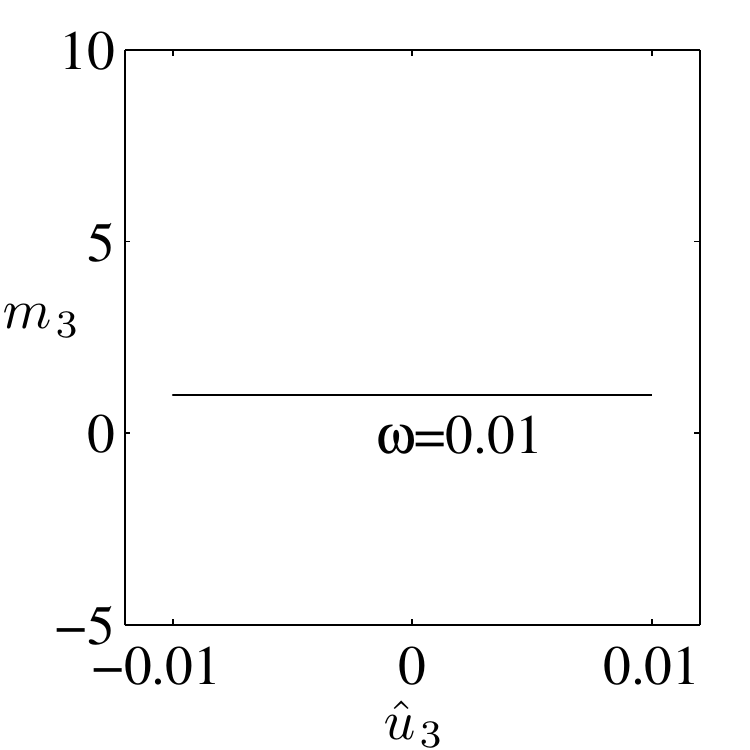}\vspace*{-0.2cm}\\
   \small{(g)}&\small{(h)}&\small{(i)}\\
\end{tabular}
{\caption{\label{fighysteresisLLphysicalcontrol}  Input--output curves for the  Landau-Lifshitz equation described in (\ref{eqcontrolledLLphysical}) to (\ref{eqconstraint}). Parameters are set to $L=1$, $\nu=0.02$ and $x=1$ with periodic input $0.01\cos(\omega t)$ for each output $m_i(x,t)$, $i=1,2,3$ and frequencies $\omega=1,\, 0.1,\, 0.01$. (a)--(c)  shows $m_1(x,t)$ with initial magnetization $\mathbf m_0(x)=\left(1,0,0\right),$ input $\mathbf{\hat{u}}(t)=(0.01\cos(\omega t),0,0)$,  and $\mathbf r(t)=(1,0,0)$; (d)--(f) shows $m_2(x,t)$  with initial magnetization  $\mathbf m_0(x)=\left(0,1,0\right)$, input $\mathbf{\hat{u}}(t)=(0,0.01\cos(\omega t),0)$ and $\mathbf r(t)=(0,1,0)$; (g)--(i)  shows $m_3(x,t)$ with initial magnetization  $\mathbf m_0(x)=\left(0,0,1\right),$ input $\mathbf{\hat{u}}(t)=(0,0,0.01\cos(\omega t))$, and $\mathbf r(t)=(0,0,1)$. }}
\end{figure}

Because hysteresis is characterized by multiple stable equilibrium points \cite{Chow2014_ACC,Morris2011}, the ability to control the stability of equilibrium points implies the ability to control hysteresis. Such a control for the Landau--Lifshitz equation is given in (\ref{equ}). This is a possible avenue for future exploration.

%%%%%%%%%%%%REFERENCE%%%%%%%%%%%%%%%%%%%%%
\bibliographystyle{plain}%used in the automatica sample file
\bibliography{ref}

\end{document}